\documentclass[a4paper]{article}

\usepackage{makeidx, calc, amsmath, amsthm, a4, latexsym, amssymb, layout, subfigure}
\usepackage[dvips]{graphicx}

\newcommand{\ssup}[1] {{\scriptscriptstyle{({#1}})}}

\newcommand{\me}{{\mathbf E}}
\renewcommand{\mp}{{\mathbf P}}

\newtheorem{thm}{Theorem}[section]
\newtheorem{lemma}[thm]{Lemma}
\newtheorem{corollary}[thm]{Corollary}
\newtheorem{prop}[thm]{Proposition}

\theoremstyle{definition}

\newtheorem{rem}[thm]{Remark}

\newcommand{\be}[1]{\begin{equation}\label{#1}}
\newcommand{\ee}{\end{equation}}
\newcommand{\ba}{\begin{array}}
\newcommand{\ea}{\end{array}}

\newcommand{\R}{\mathbb{R}}

\newcommand{\N}{\mathbb{N}}
\newcommand{\Q}{\mathbb{Q}}

\newcommand{\E}{\mathbb{E}}
\newcommand{\p}{\mathbb{P}}

\newcommand{\calF}{\mathcal{F}}
\newcommand{\calG}{\mathcal{G}}

\newcommand{\calV}{\mathcal{V}}

\newcommand{\esssup}{\mathrm{ess}  \sup}

\newcommand{\1}{\mathbf{1}}
\newcommand{\inv}{^{-1}} 

\newcommand{\la}{\lambda}

\newcommand{\eps}{\varepsilon}

\renewcommand{\phi}{\varphi}

\newcommand{\specialh}{h}
\newcommand{\tV}{\tilde{V}}

\newcommand{\ra}{\rightarrow}
\newcommand{\da}{\downarrow}

\newcommand{\Teps}{T^{\ssup \eps}}
\newcommand{\Ttilde}{\tilde{T}}
\newcommand{\tT}{\tilde{T}}
\newcommand{\st}{\textbf{\textsf{SpinedTrees}} }
\newcommand{\ep}{\mathbb{P}}
\newcommand{\eE}{\mathbb{E}}
\newcommand{\bias}{*} 

\begin{document}

\begin{center}
{\LARGE \bf Minimal supporting subtrees for the free energy of polymers on disordered trees}
\vspace{0.5cm}\\
 
\textsc{Peter M\"orters} and \textsc{Marcel Ortgiese}\\
Department of Mathematical Sciences\\
University of Bath\\
Bath BA2 7AY\\
United Kingdom
\vspace{0.2cm}\\

\end{center}

\vspace{0.5cm}

\begin{quote}{\small {\bf Abstract: }
We consider a model of directed polymers on a regular tree
with a disorder given by independent, identically distributed 
weights attached to the vertices. For suitable weight distributions
this model undergoes a phase transition with respect to its localization 
behaviour. We show that, for high temperatures, the free energy is supported by a random
tree of positive exponential growth rate, which is strictly smaller than that 
of the full tree. The growth rate of the minimal supporting subtree 
is decreasing to zero as the temperature decreases to the 
critical value. At the critical value and all lower temperatures, a single polymer 
suffices to support the free energy. Our proofs rely on elegant martingale methods adapted
from the theory of branching random walks.
}
\end{quote}

\section{Introduction and main results}

In this paper we give a detailed study of the phase transition arising from the presence of a 
random disorder in the very basic model of polymers on disordered trees introduced by Derrida 
and Spohn in~\cite{DS88}. This phase transition becomes manifest in the behaviour of the free energy, 
see~\cite{BPP93}, but also in the localization behaviour of the model, which is measured in terms of 
the size of the smallest subtree supporting the free energy. The model can be seen as a 
mean-field version of the popular model of a directed polymer in random environment, where most of 
the questions settled here are still open. For a survey of directed polymers see~\cite{CSY04}, and 
note also the recent important results obtained by Comets and Yoshida in~\cite{CY06}.
\smallskip

For a precise description of the polymers on disordered trees, let $d\ge 2$ and $T$ be a $d$-ary tree such that, 
starting from an initial ancestor in generation~$0$, the \emph{root}~$\rho$, each vertex has exactly
$d$~\emph{children}. A \emph{polymer} is a finite or infinite self-avoiding path started in the root. 
We write $|v|$ for the 
generation of a vertex~$v$ and denote by $T_n = \{ v \in T \, : \, |v| =n\}$ the set of vertices in 
the $n$th~generation. Each $v\in T_n$ can be identified with the unique path $(v_0,v_1, \ldots, v_n)$ of 
its ancestors from $v_0=\rho$ to $v_n = v$, and thus represents a polymer of length~$n$.
\smallskip

Consider a non-degenerate random variable $V$, which has all exponential moments, i.e.
\[ \me[ e^{\beta V}] < \infty \quad\mbox{ for all } \beta \in \R \, . \]
Then we introduce the \emph{random disorder} $\calV = (V(v) \, : \, v \in T)$ as a collection of independent
distributed weights with the same distribution as $V$ attached to the vertices of the tree. 
For a finite length polymer $v \in T_n$ we introduce the \emph{Hamiltonian} 
\[ H_n(v) = -\sum_{j=1}^n V(v_j) \, . \]
The \emph{polymer measure} or \emph{finite volume Gibbs measure} $\mu_n^{\ssup \beta}$ on $T_n$ is defined by
\[ \mu_n^{\ssup \beta}  = \frac{1}{Z_n(\beta)} \sum_{v\in T_n} e^{-\beta H_n(v)}
\delta_{v} \, , \]
where $\beta>0$ is the inverse temperature and 
the normalising constant $Z_n(\beta)$ is the \emph{partition function} defined as
\[ Z_n(\beta) = \sum_{v \in T_n} e^{\beta \sum_{j=1}^n V(v_j)} \, . \]

Polymers of infinite length can be represented as a sequence $(\xi_0, \xi_1, \xi_2, \ldots)$ 
of vertices, such that $\xi_{n}$ is a vertex in the $n$th generation, and moreover a child of $\xi_{n-1}$. 
Such sequences are called~\emph{rays} and the set of all rays constitute the \emph{boundary} of the tree, 
denoted by~$\partial T$. We equip the boundary $\partial T$ with the metric
$d(\xi, \eta ) = \exp(- \sup\{ n \geq 0 \, : \, \xi_n =\eta_n \})$, for $\xi, \eta \in \partial T$,
which makes $\partial T$ a compact metric space. 
\smallskip

We  first review some of the basic properties of the model. Roughly speaking, one should
expect that the behaviour of the polymer depends 
on the inverse temperature parameter~$\beta$ in the following manner: If $\beta$ is small, 
we are in an \emph{entropy-dominated} regime, where the disorder has no  big influence and 
limiting features are largely the same as in the case of a uniformly distributed polymer. 
For large values of $\beta$ we may encounter an \emph{energy-dominated} regime where, 
due to the disorder, the phase space breaks up into pieces separated by free energy barriers. 
Polymers then follow specific tracks with large probability, an effect often
called \emph{localisation}.\smallskip
 
The mathematical analysis of polymers on disordered trees is based on the family
of martingales $(M^{\ssup \beta}_n \colon n\ge 0)$ defined by
\[ M^{\ssup \beta}_n = e^{-n(\la(\beta)+\log d)} Z_n(\beta) \, , \qquad \mbox{ for }
n \geq 0 \, , \]
where
\[ \la(\beta) = \log \me e^{\beta V} \, , \]
is the logarithmic moment generating function of~$V$. It is easy to check that, for any
$\beta\ge 0$, $(M^{\ssup \beta}_n \colon n\ge 0)$ is a martingale with respect to the
filtration $\calF_n = \sigma( V(v) \colon |v| \leq n )$, $n \ge 0$.
Since the martingale is non-negative, its limit
$M^{\ssup \beta} = \lim_{n \ra \infty} M^{\ssup \beta}_n$ exists almost surely.
An easy application of Kolmogorov's zero-one law shows that $\p \{ M^{\ssup \beta} = 0 \} \in \{ 0,1\}$. 
\smallskip

Define the function 
$$ f(\beta) = \la(\beta) + \log d - \beta \la'(\beta) \qquad \mbox{ for $\beta \ge 0$.} \\[1mm]$$
From the strict convexity of~$\lambda$, we infer that $f(\beta)<\log d$ for all $\beta>0$.
We shall check in Lemma~\ref{criterion_for_root_of_f} below that $f$ has a positive root unless 
the law of~$V$ is bounded from above with an atom of mass $\ge \frac 1d$ at its essential supremum.
Let $\beta_{\rm c}$ be the positive root, if it exists, and $\beta_{\rm c}=\infty$ otherwise.
Kahane and Peyri{\`e}re~\cite{KP76} and Biggins~\cite{Bi77} show that 
\[ \ba{ll} M^{\ssup \beta} > 0 \ \mbox{almost surely,} & \mbox{if } \beta < \beta_{\rm c}, \\
M^{\ssup \beta} = 0 \ \mbox{almost surely,} & \mbox{if } \beta \ge \beta_{\rm c} \, . \ea \]
In particular, they show that $\E [M^{\ssup \beta}] = 1$ if and only if $\beta < \beta_{\rm c}$.
In this paper, we are especially interested in the free energy, defined as
\[ \phi(\beta) = \lim_{n \ra \infty} \frac{1}{n} \log Z_n(\beta) \, . \]
It turns out that $\beta_{\rm c}$, if finite, is the critical parameter for a change in the 
qualitative behaviour of the free energy. Indeed,
\be{free_energy} \phi(\beta) = \left\{ \ba{cc} \la(\beta)+ \log d & \mbox{if } \beta \leq \beta_{\rm c} \, , \\[2mm] 
\frac{\beta}{\beta_{\rm c}}\, \big( \la(\beta_{\rm c})+\log d\big) & \mbox{if } \beta > \beta_{\rm c} \,. \ea \right. \ee
This result was stated in~\cite{DS88} and proved for a continuous time analogue. 
An elementary proof, based on the study of the martingales~$(M_n^{\ssup\beta}\, : \, n \geq 0)$, 
can be found in~\cite{BPP93}. We observe that at the critical temperature $1/\beta_{\rm c}$ the model 
undergoes a phase transition and, for low temperatures, it is frozen in the ground state. 
The two phases are often called the \emph{weak disorder} phase ($\beta<\beta_{\rm c}$), 
and the \emph{strong disorder} phase ($\beta>\beta_{\rm c}$). 
See Figure~\ref{fig_free_energy} for an illustration.
\smallskip

\begin{figure}[htbp]
\begin{center}
\includegraphics{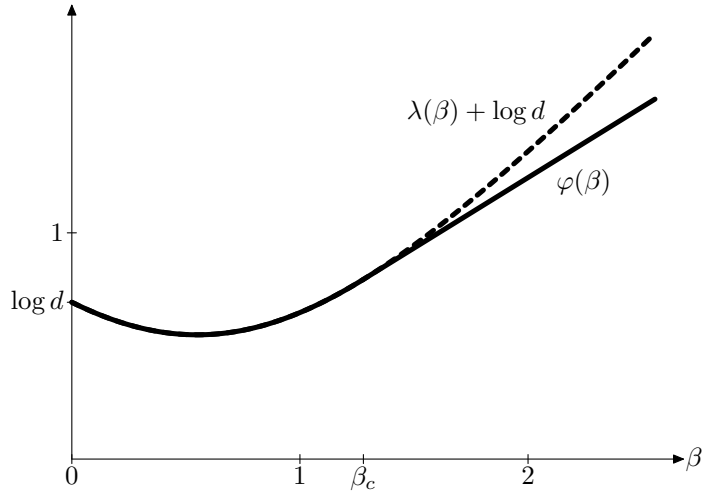}
\caption{The free energy for the model when $\mp \{ V =1 \} = 1/4 = 1 - \mp \{ V = -1\}$ and $d = 2$. }
\label{fig_free_energy}
\end{center}
\end{figure}

In the \emph{weak disorder phase} the form of \eqref{free_energy} seems to suggest that, asymptotically,
each of the $d^n$~polymers~$v \in T_n$ contributes a summand
$$\E [ e^{\beta \sum_{j=1}^n V(v_j)} ] = \exp\big[ n \la(\beta) \big]$$
to the partition function~$Z_n(\beta)$, and therefore the finite volume Gibbs measure 
does not localize on a significantly smaller subset of $T_n$. 
However, our first main result shows that this picture is \emph{wrong} and already a 
vanishing proportion of paths make a significant contribution to the free energy. 
These paths can be chosen to be the vertices of a tree, which we call a \emph{minimal supporting subtree}.

\begin{thm}\label{smaller_tree} 
Let $0<\beta < \beta_{\rm c}$ so that we are in the weak disorder phase.
\begin{itemize}
\item[(a)] Almost surely, there exists a tree $\tT \subset T$ of growth rate
\[ \lim_{n\ra \infty} \frac{1}{n} \log |\tT_n| =  f(\beta)<\log d, \] such that
\[ \lim_{n \ra \infty} \tfrac{1}{n} \log \sum_{v \in \tT_n} e^{\beta \sum_{j=1}^n V(v_j)} = \phi(\beta) \, . \]
	\item[(b)] Almost surely for every sequence $(A_n)_{n \geq 1}$ of non-empty subsets $A_n \subset T_n$ of the 
	vertices in the $n$th generation satisfying
	\[ \limsup_{n\ra \infty} \frac{1}{n} \log |A_n| < f(\beta)  \] 
	we have that 
	\[ \limsup_{n \ra \infty} \frac 1n \log \sum_{v \in A_n} e^{\beta \sum_{j=1}^n V(v_j)} < \phi(\beta) \, . \]
\end{itemize}
\end{thm}
\pagebreak[2]

\begin{rem}\ \\[-5mm]
\begin{itemize}
\item Loosely speaking, if $0<\beta < \beta_{\rm c}$, vertices in generation~$n$ of the minimal 
supporting subtree typically contribute a summand $\exp(n \beta \lambda'(\beta))$ to the partition 
function $Z_n(\beta)$. As the number of such vertices is of order $\exp(n f(\beta))$, this is in 
line with the equation $f(\beta)+\beta \lambda'(\beta) = \phi(\beta)$.
\item The function $f$ can be interpreted as the \emph{entropy} of the system. 
Its r\^ole as a \emph{multifractal spectrum} is highlighted in~\cite{Mo08}.
\end{itemize} 
\end{rem}
\smallskip

At the critical temperature, the growth rate of the minimal supporting subtree hits zero. This suggests that in the
\emph{strong disorder phase} a subexponential set of polymers may support the free energy.
This is true, and our second main result even shows that a single polymer suffices.
\smallskip

\begin{thm}\label{one_ray_suffices}  If $\beta_{\rm c} < \infty$, then almost surely there
exists a ray $\xi=(\xi_0, \xi_1, \ldots)\in \partial T$ such that for 
any $\beta \geq \beta_{\rm c}$ and sets $A_n\subset T_n$ containing the vertex~$\xi_n$,
$$\lim_{n\to\infty} \frac{1}{n} \log \sum_{v \in A_n} e^{\beta \sum_{j=1}^n V(v_j)} 
= \beta \, \lim_{n\to\infty} \frac{1}{n}  \sum_{j=1}^n V(\xi_n) 
= \phi(\beta) \, .$$
\end{thm}
\smallskip

Directed polymer models are intimately related to the model of \emph{$\varrho$-percolation}
introduced by Menshikov and Zuev~\cite{MZ93}, which is considered for example 
in~\cite{KS00} and~\cite{CPV08}. Here we discuss an interesting implication of our 
results for this model.
\smallskip
 
To define $\varrho$-percolation,  given an infinite, connected graph 
and a survival parameter $p\in(0,1)$, we declare each edge independently to
be open with probability~$p$, or closed with probability~$1-p$. For $\varrho\in[p,1]$
we say that \emph{$\varrho$-percolation occurs}, if there exists an infinite self-avoiding
path, along which the asymptotic proportion of open edges is at least~$\varrho$.
Our result gives a sharp criterion for the occurrence of $\varrho$-percolation on
regular trees.

\begin{thm}\label{thm_rho_percolation} For $\varrho \in (0,1]$, 
\[ \varrho\mbox{-percolation occurs almost surely} \quad \iff \quad p \geq p_{\rm c} \, , \]
where $p_{\rm c}=\frac 1d$ if $\varrho=1$, and otherwise $p_{\rm c}$ is the  unique solution 
in the interval $(0,\varrho)$ of the equation
\[ p_{\rm c}^\varrho (1-p_{\rm c})^{1-\varrho} d = \varrho^\varrho (1-\varrho)^{1-\varrho} \, . \]
\end{thm}

\begin{rem}\ \\[-5mm]
\begin{itemize}
\item The most interesting fact here is that $\varrho$-percolation occurs at criticality,
a phenomenon which we conjecture to hold for $\varrho$-percolation on arbitrary trees.  
\item If $\varrho = 1$, then the critical $p$ value is $\frac{1}{d}$ which is the same as for 
classical percolation on a $d$-ary tree. However, unlike in the classical case, $1$-percolation 
occurs at criticality. Our proofs also show that in the case $p>p_{\rm c}$ the Hausdorff dimension of 
the set of rays surviving $\varrho$-percolation agrees with that of the boundary of the surviving tree
in classical percolation, provided the latter is non-empty.
\end{itemize}
\end{rem}

The remainder of this paper is structured as follows. In Section~\ref{properties_of_f} we review some of the basic 
properties of the function $f$. In Section~\ref{ergodic_theory} we focus on the weak disorder phase and develop some basic ergodic theory of weighted trees, which enables us to construct and explore some properties of the infinite volume Gibbs measures. We also give an estimate on the number of polymers of length~$n$ for which the Hamiltonian is unusually small in terms of a coarse multifractal spectrum. Using this, we prove Theorem~\ref{smaller_tree} in Section~\ref{sect_smaller_tree}. More subtle techniques are required to discuss the critical case and tackle Theorem~\ref{one_ray_suffices}. These are developed in Section~\ref{localisation_critical_tree}. Finally, in Section~\ref{rho_percolation} we translate our results to the model of $\varrho$-percolation and complete the proof of 
Theorem~\ref{thm_rho_percolation}.

\section{Preliminaries}\label{properties_of_f}

In this section, we review some of the properties of the  function $f$. In particular, 
we establish a necessary and sufficient condition for $f$ to have a positive root, see Lemma~\ref{criterion_for_root_of_f}. We also prove a result about the minimum of the 
Hamiltonian taken over the vertices in the $n$th generation, see Lemma~\ref{max_min_H_n}.

We require the Legendre-Fenchel transform $\la^*$ of $\la$ defined as
\[ \la^*(\alpha) = \sup_{\beta \in \R} \{ \alpha \beta - \la(\beta) \} \, , \]
see Figure~\ref{legendre_transform_lambda} for an illustration.

\begin{figure}[htbp]\label{legendre_transform_lambda}
\centering 
\includegraphics[height=6cm]{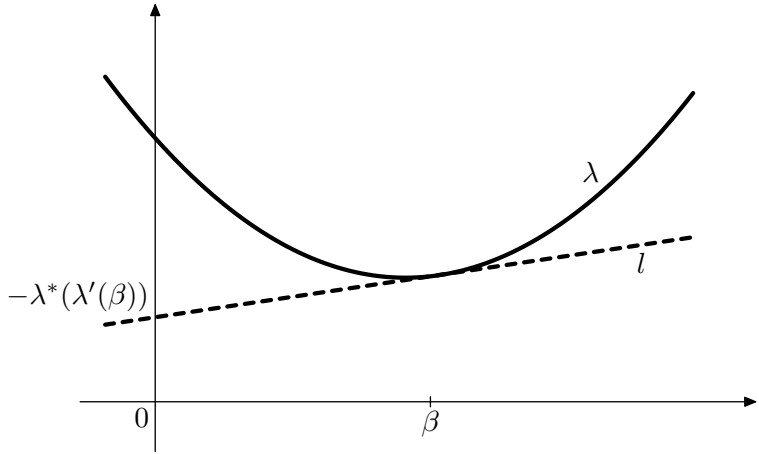} 
\caption{The \emph{Legendre-Fenchel transform} of $\la$. Let $\alpha \in \R$. If $l$ is the unique line of support of $\la$ at $\beta$ with slope $\alpha$, then  $- \la^*(\alpha)$ is equal to the $y$-coordinate of the intersection point of $l$ with the vertical axis.}
\end{figure}

The next result, which can be found in~\cite{Co05}, gives us a necessary and sufficient condition for $f$ to have a positive root. 
\smallskip
\pagebreak[3]

\begin{lemma}\label{criterion_for_root_of_f} $f$ has a 
positive root if and only if 
\begin{itemize}
\item \emph{either} $V$ is unbounded,
\item \emph{or} $w:= \esssup V$ is finite and $\mp\{V=w\}<\frac 1d$.
\end{itemize} 
\end{lemma}

\begin{proof} 
Using the Legendre-Fenchel transform, we find that 
\be{definition_f} f(\beta) = \log d + \la(\beta) - \beta \la'(\beta) = \log d - \la^*(\la'(\beta)) \, . \ee
Since $f(0) = \log d$ and $f$ is strictly decreasing and continuous, $f$ has a positive root if and only if $\lim_{\beta \ra \infty} f(\beta) < 0$. It is well-known that
\be{lambda_prime}
\la'(\beta) = \frac{\me \big[V e^{\beta V}\big]}{\me \big[ e^{\beta V}\big]} \ra \esssup V \, .
\ee
Therefore, if $\esssup V = \infty$, then $\la'(\beta) \ra \infty$,
which implies that $f(\beta) \ra - \infty$, so that $f$ has a positive root.

Now suppose that $w:= \esssup V < \infty$. Using $\la'(\beta) \ra w$ and the lower semi-continuity of $\la^*$,
\[ \begin{aligned} \lim_{\beta \ra \infty} \la^*( \la'(\beta)) & = \la^*( w)= \sup_{\beta} \big( \beta w - \log \me\big[ e^{\beta V} \big] \big)
\\ & = - \inf_\beta \, \big(\log ( \mp \{ V = w \}  + \me[ \1{\{V < w\}} \, e^{\beta( V - w)}] ) \big) 
= - \log \mp \{V = w\} \, . 
\end{aligned} \]
So in particular, by~(\ref{definition_f}), $\lim_{\beta \ra \infty} f(\beta) = \log d + \log \mp \{ V = w\}$.
Therefore, if $\mp \{ V = w \} < \frac{1}{d}$, then $\lim_{\beta \ra \infty} f(\beta) < 0$, 
i.e. $f$ has a positive root. Conversely, if $\mp \{ V = w\} \geq \frac{1}{d}$, then  
$\lim_{\beta \ra \infty} f(\beta) \geq 0$ implying that $f(\beta) > 0$ for all $\beta \geq 0$.
\end{proof}

\begin{figure}[htbp]\label{f_of_alpha}
\centering 
\subfigure[$V$ with $\mp \{ V =1\}= 1 - \mp \{ V = 0\}  < 1/d$.]{ 	
		\includegraphics[height=6cm, width=6cm]{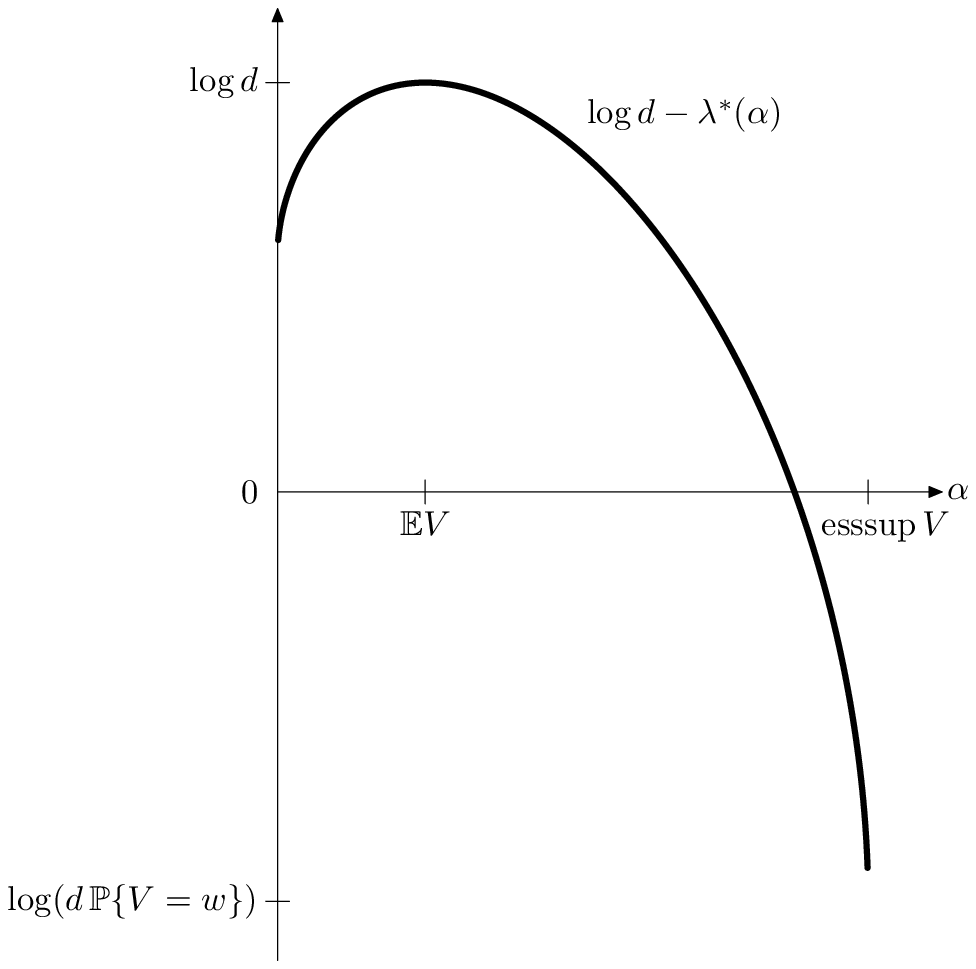}
		\label{f_V_binary_1}  }
\subfigure[$V$ with $\mp \{ V =1\} = 1 - \mp \{ V = 0\} \geq 1/d$.]{ 	
		\includegraphics[height=6cm, width=6cm]{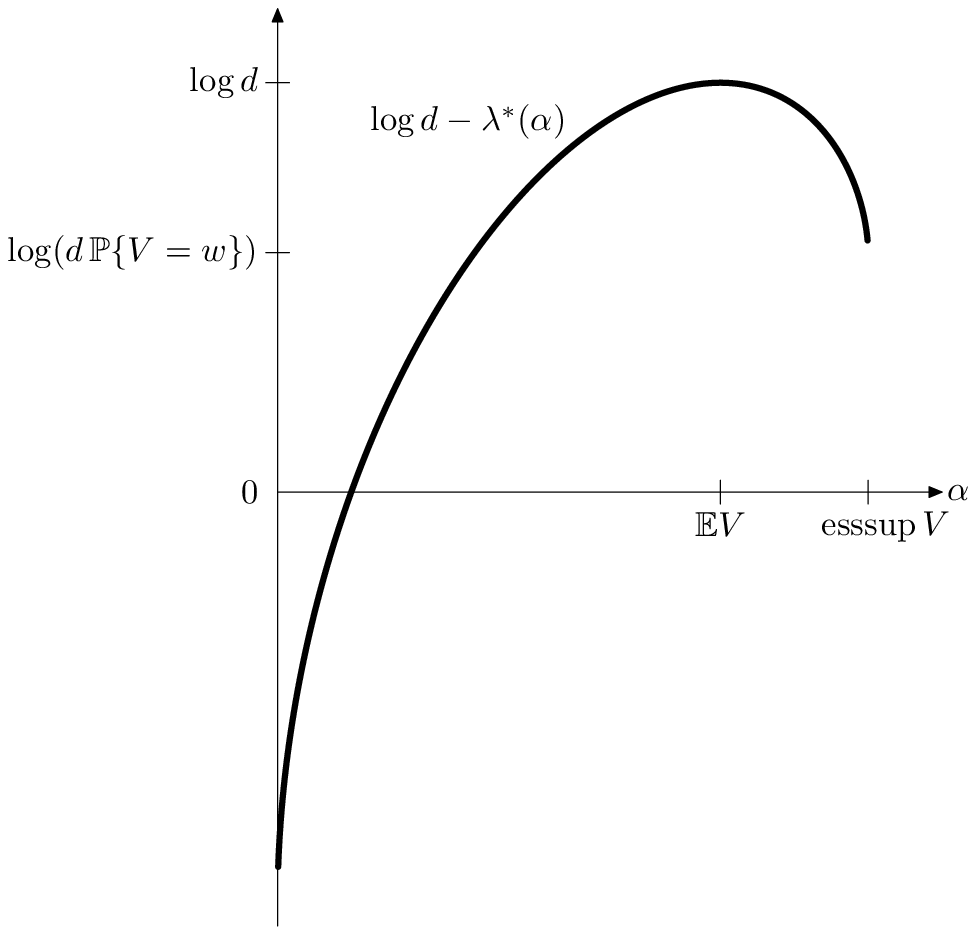}
		\label{f_V_binary_2} }
\subfigure[$V$ uniformly distributed on \mbox{$[0,1]$}.]{ 	
		\includegraphics[height=6cm, width=6cm]{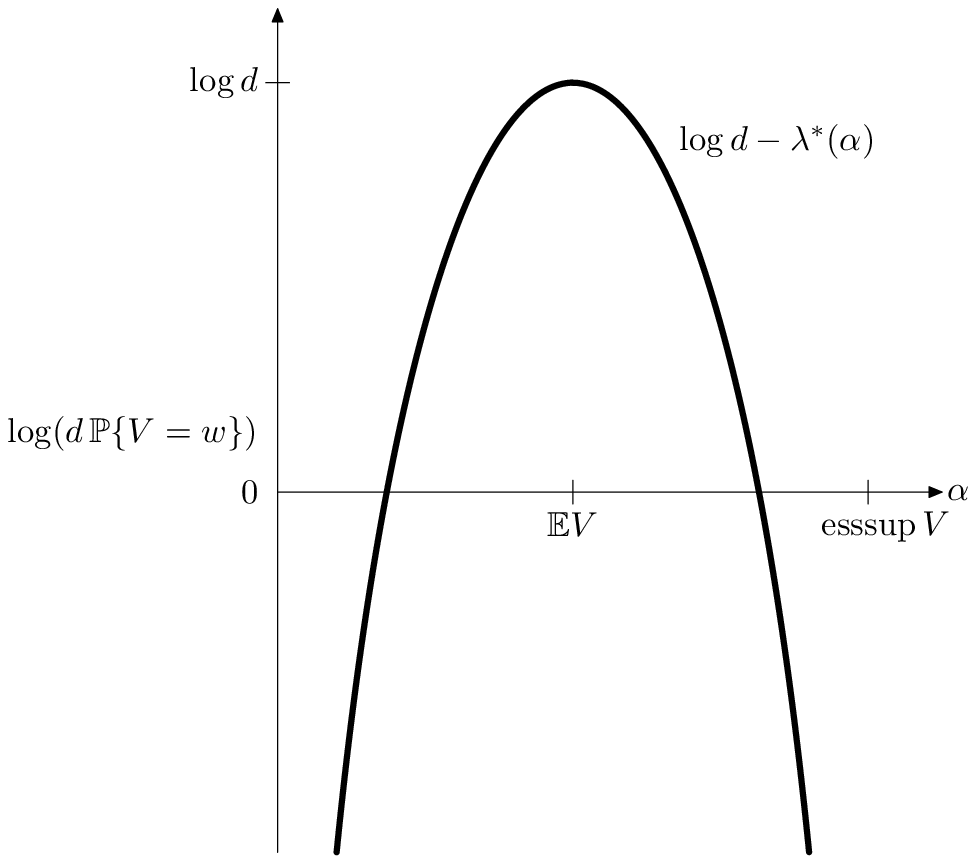}
		\label{f_V_uniform} }
\subfigure[$V$ standard normally distributed.]{ 	
		\includegraphics[height=6cm, width=6cm]{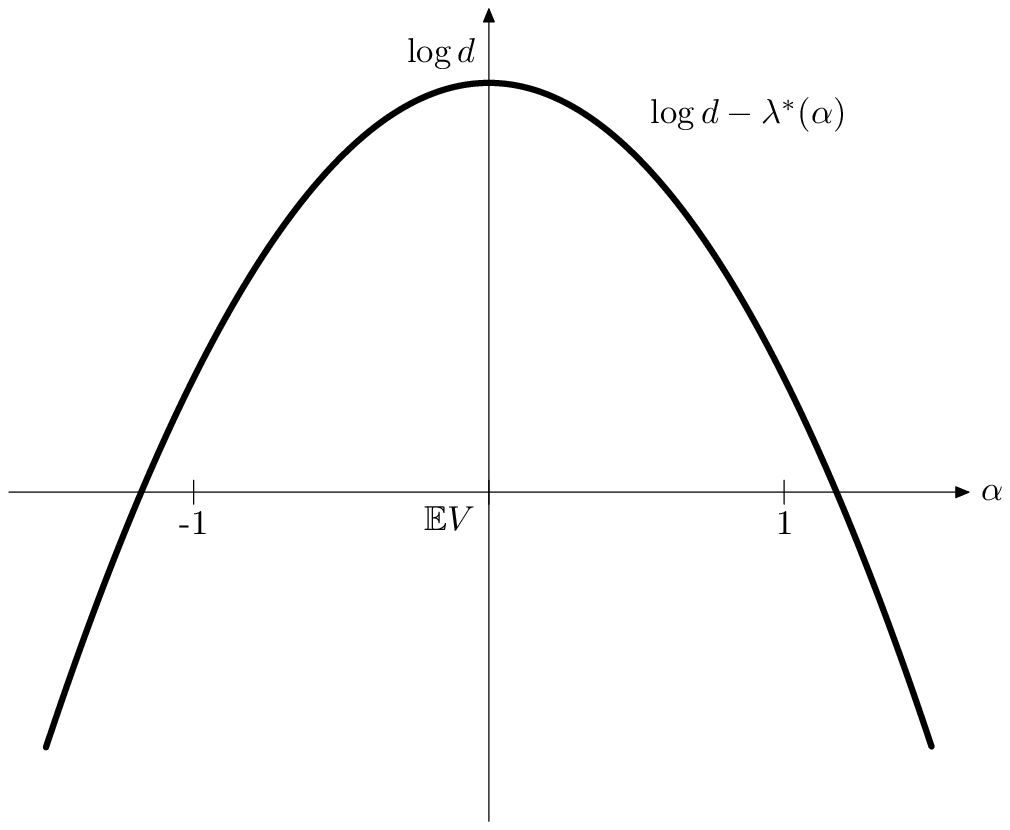}
		\label{f_V_normal} }
\caption{The function $\alpha \mapsto \log d - \la^*(\alpha)$ for four typical cases. Writing $w = \esssup V$, 
Figure (a) shows the
case that $V$ is bounded, but $0 < \mp \{ V = w\}  < \frac{1}{d}$, whereas in (b) $V$ is bounded, but
$\mp\{ V = w \} \geq \frac{1}{d}$, in (c) $V$ is still bounded, but $\mp\{V  = w\} = 0$. Finally, 
in (d), $V$ is unbounded.
}
\end{figure}

By (\ref{lambda_prime}), it makes sense in the case $\beta_{\rm c} = \infty$
to define $\la'(\beta_{\rm c}) = \esssup V$. With this convention, we can prove the
following lemma about the minimum of the Hamiltonian taken over the vertices
in the $n$th generation. 

\begin{lemma}\label{max_min_H_n} We have 
\[ - \lim_{n \ra \infty} \frac{1}{n} \min_{v\in T_n} H_n(v) = 
\lim_{n \ra \infty} \frac{1}{n} \max_{v\in T_n} \sum_{j=1}^n V(v_j) = 
\lim_{\beta\ra\infty}
\frac{\phi(\beta)}{\beta} = \la'(\beta_{\rm c}) \, . \]
\end{lemma}

\begin{proof} The first equality follows from the definition of the Hamiltonian.
Clearly, for $\beta > 0$,
\[ \exp \Big\{ \beta \max_{v\in T_n} \sum_{j=1}^n V(v_j) \Big\} \leq Z_n(\beta) \leq d^n 
\exp \Big\{ \beta \max_{v\in T_n} \sum_{j=1}^n V(v_j) \Big\} \, . \]
Hence, it follows that
\be{max_free_energy} \frac{1}{n \beta} \log Z_n(\beta) - \frac{1}{\beta} \log  d  \leq 
\max_{v\in T_n} \frac{\sum_{j=1}^n V(v_j)}{n} \leq \frac{1}{n\beta} \log Z_n(\beta) \, . \ee
If $\beta_{\rm c} < \infty$, then we know from~(\ref{free_energy}) that
\[ \lim_{\beta \ra \infty} \frac{\phi(\beta)}{\beta} = \frac{\la(\beta_{\rm c}) + \log d}{\beta_{\rm c}} = \la'(\beta_{\rm c}) \,  , \]
by definition of $\beta_{\rm c}$. 
Moreover, if $\beta_{\rm c} = \infty$, then again by~(\ref{free_energy}), we know that 
$\phi(\beta) = \la(\beta) + \log d$ for all $\beta>0$. Also, the Legendre-Fenchel transform
of $\la$ satisfies
$\la^*(\la'(\beta)) = \la'(\beta)\beta - \la(\beta)$ and the proof of Lemma~\ref{criterion_for_root_of_f}
shows that $\lim_{\beta \ra \infty} \la^*(\la'(\beta)) = -\log \mp \{ V = w\}$. 
Therefore, we can conclude that
\[ \lim_{\beta \ra \infty} \frac{\phi(\beta)}{\beta} = \lim_{\beta \ra \infty} \frac{\la(\beta)}{\beta}
= \lim_{\beta\ra \infty} \Big( \la'(\beta) - \frac{\la^*(\la'(\beta))}{\beta}\Big) = \esssup V = \la'(\beta_c) \, , \]
by~(\ref{lambda_prime}) and the convention $\la'(\beta_{\rm c}) = \esssup V$.
Hence, in either case, letting first $n \ra \infty$ and then $\beta \ra \infty$ in~(\ref{max_free_energy})
yields the statement of Lemma~\ref{max_min_H_n}.
\end{proof}


\section{Ergodic theory and the multifractal spectrum}\label{ergodic_theory}

In the next two sections we  concentrate on the weak disorder phase, in other words
we  assume that $\beta < \beta_{\rm c}$, so that the martingale limit $M^{\ssup \beta}$ is positive.

\subsection{Ergodic theory on weighted trees}\label{subsection_ergodic_theory}

We  develop the ergodic theory for a tree with attached weights in 
analogy to the  ergodic theory on Galton-Watson trees developed 
by Lyons, Pemantle and Peres in~\cite{LPP95}.
We take advantage of the fact that, in the weak disorder regime, the martingale
convergence can be used to construct the infinite-volume Gibbs measure 
on the boundary of the tree. For this purpose, we extend a finite length polymer $v = (v_0,\ldots, v_n)$ 
to an infinite length polymer  $v^+\in\partial T$ by defining $v_{i+1}$ 
to be the left-most child of $v_i$ for all $i \geq n$. This enables us to interpret the finite volume
Gibbs measures $\mu_n^{\ssup \beta}$ as probability measures on the boundary~$\partial T$ using the convention
$\mu_n^{\ssup \beta}(v^+)=\mu_n^{\ssup \beta}(v)$ for any $v\in T_n$. We will frequently use this identification in the sequel.

For a vertex $v\in T_n$, let $B(v) = \{ \xi \in \partial T \, : \, \xi_n = v \}$
and let $T(v)$ be the subtree consisting of all vertices that have $v$ as an ancestor, with $v$ as a root.
Then we can define the infinite-volume Gibbs measure $\mu^{\ssup \beta}$ by
\[ \mu^{\ssup \beta} (B(v)) := e^{\beta \sum_{j=1}^n V(v_j) - n (\la(\beta)+ \log d )} \frac{M^{\ssup \beta}(v)}{M^{\ssup \beta}}  \, \]
where $M^{\ssup \beta}(v)$ is defined as the almost sure limit of
\[ M^{\ssup \beta}_n(v) 
=   \sum_{w \in T_n(v)} \exp \Big( \beta \sum_{j=1}^n V(w_j) - n (\la(\beta)+\log d )\Big) \, , \]
which exists since $(M^{\ssup \beta}_n(v), n\geq 0)$ and $(M^{\ssup \beta}_n, n\geq 0)$ have
the same law. Then, we see that almost surely for $v$ such that 
$|v| = k$, as $n \ra \infty$,
\[\begin{aligned}  \mu_n^{\ssup \beta}(B(v)) & = \frac{1}{Z_n(\beta)} \, e^{\beta \sum_{j=1}^k V(v_j)} \sum_{w \in T_{n-k}(v)}  
e^{\beta \sum_{j= 1}^{n-k} V(w_j)}  \\
& = e^{\beta \sum_{j=1}^k V(v_j) - k (\la(\beta)+\log d ) } \, \frac{M_{n-k}^{\ssup \beta}(v)}{M_n^{\beta}} \ra  \mu^{\ssup \beta} (B(v))\, , 
\end{aligned} \]
in other words, almost surely, $\mu^{\ssup \beta}_n$ converges weakly to $\mu^{\ssup \beta}$.

The central result of this section is the following proposition.

\begin{prop}\label{limit_sum_V_local_dimension}
If $\beta<\beta_{\rm c}$, for $\p$-almost every disorder and $\mu^{\ssup \beta}$-almost 
every path~$\xi \in \partial T$, 
\be{average_converges_to_alpha}
\lim_{n \ra \infty} \frac{1}{n} \sum_{j=1}^n V(\xi_j) = \la'(\beta) \, , \ee
and
\be{local_dimension_of_mu} \lim_{n \ra \infty} -\frac{1}{n} \log \mu^{\ssup \beta}(B( \xi_n)) = f(\beta)\, . \ee
\end{prop}

Let $\st = \{ (\calV, \xi) \, \colon \, \calV=(V(v) \colon v\in T), \, \xi \in \partial T\}$  
be the space of weights attached to the vertices of the $d$-ary tree with marked spine, endowed
with the product topology. For any vertex $w\in T$ we denote by $\calV(w) = (V(v): v \in T(w))$ 
the family of weights  on the tree $T(w)$. There is a canonical shift 
$$\theta \colon \st \ra \st, \quad \theta(\calV, \xi)=\big( \calV(\xi_1), (\xi_1, \xi_2, \ldots) \big).$$ 
Our aim is to show that $\theta$ is a measure-preserving transformation with respect to the
measure
\[ \nu (d (\calV , \xi)) = \mu_\calV^{\ssup \beta}(d \xi) \, M^{\ssup \beta}_\calV \, \p (d\calV) \, , \]
where the subscript $\calV$ indicates the dependence of $\mu^{\ssup \beta}_\calV$ and 
$M_\calV^{\ssup \beta}$ on the underlying disorder.

\begin{lemma}\label{shift_measure_preserving} The shift $\theta$ is $\nu$-preserving. \end{lemma}

\begin{proof}
Let $A$ be a Borel set in \st. Then,
\be{la_theta_inv} \begin{aligned} \nu(\theta\inv A) & = \int \1_{\theta\inv A} (\calV, \xi) 
\, \mu^{\ssup \beta}_\calV (d\xi) \, M^{\ssup \beta}_\calV \p(d\calV) \\ & =  
\int \sum_{|v| =1} \1_{\{\xi_1 = v\}}(\xi) \, \1_A (\calV(v), (v, \xi_2, \xi_3, \ldots)) 
\, \mu^{\ssup \beta}_\calV (d\xi) \,  M^{\ssup \beta}_\calV\, \p(d\calV) \, . \end{aligned} \ee
For any vertex $v=(v_0,\ldots,v_n)\in T$ we interpret $\partial T(v)$ as a subset of $\partial T$ by identifying
$(v,\xi_1, \xi_2, \ldots)\in\partial T(v)$ with $(v_0, \ldots, v_n, \xi_1, \xi_2, \xi_3, \ldots)\in \partial T$.
Hence, for $v\in T$ and $U \subset \partial T(v)$, 
\[ \mu^{\ssup \beta}_{\calV(v)}(U)  = \frac{\mu^{\ssup \beta}_\calV (U)}{\mu^{\ssup \beta}_\calV(B(v))} \, . \] 
Hence,  recalling that $\mu^{\ssup \beta}_\calV(B(v)) = e^{\beta V(v) - \la(\beta)-\log d } \frac{M^{\ssup \beta}_{\calV(v)}}{M^{\ssup \beta}_\calV}$, and using independence of the weights,
\[ \begin{aligned} \nu( & \theta\inv  A)  = 
\int \sum_{|v|=1} \int \1_A(\calV(v), (v,\xi_2,\ldots))  \, \mu^{\ssup \beta}_\calV(B(v)) \,
\mu^{\ssup \beta}_{ \calV(v)}(d (v,\xi_2, \ldots))\,  M^{\ssup \beta}_\calV\,  \p(d\calV) \\
& =  \frac{1}{d}\sum_{|v| = 1} \int  e^{\beta V(v)- \la(\beta)}  
\int \1_A (\calV(v), (v, \xi_2, \ldots))  \, \mu^{\ssup \beta}_{\calV(v)}(d(v,\xi_2, \ldots)) \,  \, M^{\ssup \beta}_{\calV(v)}\p(d\calV) \\
& = \me\big[e^{\beta V - \la(\beta)}\big] \, 
\int \1_A(\calV, \xi ) \mu^{\ssup \beta}_\calV(d\xi) \, M^{\ssup \beta}_\calV \, \p (d\calV)  = \nu (A) \, . 
\end{aligned} \]  
\end{proof}

\begin{lemma} The shift $\theta$ is ergodic. \end{lemma}

\begin{proof}By Proposition 15.5 in~\cite{LP05}, the shift is ergodic with respect to the measure $\nu$ 
if and only if every set $A$ of weights satisfying
\be{theta_invariant_2} \sum_{\substack{ \calV(v) \in A \\ |v| = 1}} \mu^{\ssup \beta}_\calV(B(v)) = \1_A(\calV) \quad \p\mbox{-almost surely } \, \ee
has $\p(A)\in\{0,1\}$. 
Therefore, let $A$ be a set satisfying~(\ref{theta_invariant_2}), then in particular,
\be{I_is_invariant}  \calV \in A \quad \iff \quad
\calV(v) \in A \ \mbox{ for all } v \mbox{ such that } |v| = 1 \, . \ee
By iteration, (\ref{I_is_invariant}) implies that $A$ is a tail event with respect to the i.i.d.~family
of weights. Invoking Kolmogorov's zero-one law, we can deduce that $\ep(A) = 0$ or $1$, as required.
\end{proof}

Since by the previous two lemmas $\theta$ is $\nu$-preserving and ergodic, the pointwise ergodic theorem gives us that 
for $\p$-almost every $\calV$ and $\mu^{\ssup \beta}_\calV$-almost every $\xi \in \partial T$, 

\be{limit_sum_V} \lim_{n \ra \infty} \frac 1n \sum_{j=1}^n V(\xi_j) = \nu[ V(\xi_1)] \, , \ee
where $\nu [\, \cdot \, ]$ denotes the expectation with respect to the measure $\nu$. We find
\[ \begin{aligned} \nu[ V(\xi_1)] & = \int V(\xi_1) \, \mu^{\ssup \beta}_\calV(d \xi) \, M_\calV^{\ssup \beta}
 \, \ep (d \calV) 
= \int \sum_{|v| =1 } V(v) \,\mu^{\ssup \beta}_\calV\{\xi_1 = v\}  \, M_\calV^{\ssup \beta} \, \ep(d\calV)  \\
& = \sum_{|v| = 1} \int  V(v) \, e^{\beta V(v) - \la(\beta) -\log d } \, M_{\calV(v)}^{\ssup \beta} \, \ep(d \calV) \\& 
= \me\big[ V \, e^{\beta V - \la(\beta)} \big] \, \eE\big[ M_{\calV}^{\ssup \beta} \big]
= \frac{ \me[ V e^{\beta V}]}{ \me [e^{\beta V}] } = \la'(\beta) 
\, , \end{aligned} \]
where we have used independence and the fact that $\eE [ M^{\ssup\beta}_\calV ] = 1$. Hence, we have proved
the first part of Proposition~\ref{limit_sum_V_local_dimension}.
Similarly, for the second part, note that
\be{local_dimension_decomposition} \lim_{n \ra \infty} -\frac{1}{n} \log \mu^{\ssup \beta}_\calV (B(\xi_n)) = \log d + \la(\beta) -
\lim_{n \ra \infty} \beta \frac{1}{n}\sum_{j=1}^n V(\xi_j) -  \lim_{n \ra \infty} \frac{1}{n}\log \frac{M^{\ssup \beta}_\calV(\xi_n)}{M^{\ssup \beta}_\calV} \, . \ee
Hence by the first part of Proposition~\ref{limit_sum_V_local_dimension}, it suffices to show that the second limit converges to~$0$. The following lemma from ergodic theory, which can be found for instance in~\cite[Lemma~6.2]{LPP95}, 
allows us to evaluate the last term. 

\begin{lemma}\label{ergodic_trick} If $S$ is a measure-preserving transformation on a probability space, $g$ is finite and measurable, and $g - Sg$ is bounded below by an integrable function, then $g - Sg$ is integrable with integral $0$.
\end{lemma}

Looking at $g(\calV, \xi) = \log M^{\ssup \beta}_\calV$ and using that 
$M^{\ssup \beta}_{\calV(\xi_1)} = M^{\ssup \beta}_{\calV}(\xi_1)$, we obtain
\[ \begin{aligned} g - \theta g & = \log  M^{\ssup \beta}_\calV - \log M^{\ssup \beta}_\calV(\xi_1) = -\log \mu_\calV^{\ssup \beta}(B(\xi_1)) + \beta V(\xi_1) - \la(\beta)-\log d  \\ & \geq  \beta V(\xi_1) - \la(\beta)-\log d  \, , \end{aligned} \]
where the latter is integrable. Hence by the ergodic theorem and Lemma~\ref{ergodic_trick}, for $\ep$-almost every disorder $\calV$ and 
$\mu^{\ssup \beta}_\calV$-almost every $\xi$,
\[ \lim_{n \ra \infty} \frac{1}{n} \log \frac{M^{\ssup \beta}_\calV(\xi_n)}{M^{\ssup \beta}_\calV}
= \lim_{n \ra \infty} \frac{-1}{n} \sum_{j=1}^n \log \frac{M^{\ssup \beta}_\calV(\xi_{j-1})}{M^{\ssup \beta}_\calV(\xi_{j})}
= \nu[ \log  M^{\ssup \beta}_\calV - \log M^{\ssup \beta}_{\calV}(\xi_1) ] = 0 \, . \]
Therefore~(\ref{local_dimension_decomposition}) together with the first part implies the second part of 
Proposition~\ref{limit_sum_V_local_dimension}.

\subsection{A coarse multifractal spectrum}

We  use the ergodic theory developed in the previous section to prove the following
\emph{coarse multifractal spectrum}. 

\begin{prop}\label{coarse_spectrum} For all $\alpha \geq \me V$ with $\la^*(\alpha)< \log d$, almost surely,
\[ \lim_{n \ra \infty} \frac{1}{n} \log \# \Big\{ v \in T_n \, : \, \sum_{j=1}^n V(v_j) \geq \alpha n \Big\}  = 
\log d - \la^*(\alpha)\, . \]
\end{prop}

\begin{proof} 
First note that the proof of Lemma~\ref{criterion_for_root_of_f} shows 
that $\alpha \beta - \la(\beta)$ is maximised at $\beta \in [0,\beta_{\rm c})$ 
such that $\alpha = \la'(\beta)$. For the \emph{upper bound} consider
\[ \begin{aligned} \phi(\beta) & = \lim_{n \ra \infty}\frac{1}{n} \log \sum_{v \in T_n} e^{\beta \sum_{j=1}^n V(v_j) } 
\geq \limsup_{n \ra \infty} \frac{1}{n} \log \sum_{v \in T_n} e^{\beta \sum_{j=1}^n V(v_j)} \1_{ \{\sum_{j=1}^n V(v_j) \geq \alpha n \}} \\
& \geq \alpha \beta + \limsup_{n \ra \infty} \frac{1}{n} \log \# \Big\{ v \in T_n \, : \, \sum_{j=1}^n V(v_j) \geq \alpha n \Big\}
\, . \end{aligned} \]
Now, by the expression for the free energy in~(\ref{free_energy}), we know that $\phi(\beta) = \la(\beta) + \log d $ for $\beta < \beta_{\rm c}$. Therefore, rearranging the previous display yields 
\[ \begin{aligned} \limsup_{n \ra \infty}  \frac{1}{n} \log \# \Big\{ & v \in T_n \, : \, \sum_{j=1}^n V(v_j) \geq \alpha n \Big\} \leq \phi(\beta) - \alpha \beta \\ & = \log d  + \la(\beta) - \la'(\beta) \beta = \log d  - \la^*(\alpha) \,,\end{aligned} \]
where we used the definition of $\beta$ as the maximizer of the Legendre-Fenchel transform.

For the \emph{lower bound}, recall that $\lambda^*$ is continuous on its domain
and consider $\eps >0$ small enough such that $\log  d  - \la^*(\alpha+\eps) >0$. 
In particular, we can find $0< \beta< \beta_{\rm c}$ such that $\alpha + \eps= \la'(\beta)$. 
Then, consider the set
\[ E = \Big\{ \xi \in \partial T \, : \, \lim_{n \ra \infty} \frac{1}{n} \sum_{j=1}^n V(\xi_j) = \la'(\beta) \,, 
\lim_{n\ra \infty} -\frac{1}{n} \log\mu^{\ssup\beta}(B(\xi_n)) = f(\beta) \Big\} \, . \]
By Proposition~\ref{limit_sum_V_local_dimension}, $\mu^{\ssup\beta}(E)=1$. Moreover, recalling that 
$\la'(\beta) - \eps= \alpha $, for any $k \in\N$, the set $E$ is covered by the collection
\[ \bigcup_{n = k}^\infty \bigcup_{|v| =n} \{ v \, : \, \sum_{j=1}^n V(v_j) \geq \alpha n, \ \mu^{\ssup\beta}(B(v)) \leq e^{-n (f(\beta) - \eps)} \} \, . \]
Hence, if we 
write $q = \liminf_{n \ra \infty} \frac{1}{n} \log \#\{ v \in T_n \, : \, \sum_{j=1}^n V(v_j) \geq \alpha n\}$, we obtain 
for $k$ sufficiently large
\[ \begin{aligned} 1 = \mu^{\ssup\beta} (E) & \leq \sum_{n=k}^\infty \sum_{|v|=n} \1_{\{ \sum_{j=1}^n V(v_j) \geq \alpha n \}}
\1_{\{ \mu^{\ssup\beta}(B(v)) \leq e^{-n (f(\beta) - \eps)} \}} \mu^{\ssup\beta}(B(v)) \\
& \leq \sum_{n = k}^\infty \#\{ v \in T_n \, : \, \sum_{j=1}^n V(v_j) \geq \alpha n\} \, e^{-n (f(\beta) - \eps)} 
\leq \sum_{n =k}^\infty e^{n ( q - f(\beta) + 2 \eps )} \, . \end{aligned} \]
Therefore, if $q - f(\beta) + 2 \eps < 0$, the sum on the right hand side converges, so by taking $k$ 
large enough we could make the right hand side $< 1$ contradicting $\mu^{\ssup\beta}(E) = 1$.
Thus, we conclude that $q - f(\beta) + 2 \eps \geq 0$.

Finally, we recall that $f(\beta) = \log d  + \la(\beta) - \beta \la'(\beta) = \log d  - \la^*(\alpha +\eps)$
so that we have shown that
\[ q = \liminf_{n \ra \infty} \frac{1}{n} \log \# \Big\{ v \in T_n \, : \, \sum_{j=1}^n V(v_j) \geq \alpha n \Big\}
\geq f(\beta) - 2 \eps = \log d  - \la^*(\alpha + \eps) - 2 \eps \, . \]
Therefore, recalling that $\la^*$ is continuous, we obtain the required lower bound
by letting $\eps \da 0$.
\end{proof}


\section{Localisation in the weak disorder phase}\label{sect_smaller_tree}

In this section, we  prove Theorem~\ref{smaller_tree} using the theory developed
in Section~\ref{ergodic_theory}.

\begin{lemma}\label{liminf_generation_size}
Suppose $\Ttilde \subset T$ is any subtree satisfying $\mu^{\ssup \beta}(\partial \Ttilde) > 0$, then 
$$\liminf_{n \ra \infty} \tfrac{1}{n} \log |\Ttilde_n| \geq f(\beta).$$
\end{lemma}

\begin{proof}
Using Frostman's lemma, see e.g.~Proposition 2.3 in~\cite{Fa97}, 
in combination with~\eqref{local_dimension_of_mu} we infer that the
Hausdorff dimension of $\partial \Ttilde$ must be at least~$f(\beta)$. The Hausdorff
dimension of the boundary of a tree is the logarithm of its branching rate, which is bounded from above
by the lower growth rate, see e.g.~\cite{LP05}.
\end{proof}

The next lemma enables us to choose suitable trees for Theorem~\ref{smaller_tree}(a).

\begin{lemma}\label{consequence_egorov}
Almost surely, for any $\eps>0$ there exists a subtree $\Teps \subset T$ with
$\mu^{\ssup \beta}(\partial \Teps) \geq 1 - \eps$, and a sequence $\delta_n \da 0$
such that, for every $\xi\in\partial \Teps$ and $n\ge 1$,
$$\frac 1n \sum_{j=1}^n V(\xi_j) \geq \la'(\beta) - \delta_n \qquad \mbox{ and } \qquad 
\mu^{\ssup \beta}\big( B(\xi_n) \big) \geq e^{-n(f(\beta) + \delta_n)} \, . $$
\end{lemma}

\begin{proof}
Since $\partial T$ is a complete separable metric space and $\mu^{\ssup \beta}$ is a finite measure, we know that $\mu^{\ssup \beta}$ is regular, see~\cite{Sc05}. By Egorov's theorem, see e.g.~\cite{Ash00}, we can can pick a closed subset $A_\eps\subset\partial T$ with the properties that $\mu^{\ssup \beta}(A_\eps) \geq 1 - \eps$  and the limits of Proposition~\ref{limit_sum_V_local_dimension} hold uniformly on $A_\eps$. This means that
there exists $\delta_n \da 0$ such that the displayed properties in the lemma hold. Now define 
$$\Teps = \bigcup_{\xi \in A_\eps} \bigcup_{j=0}^\infty \xi_j,$$ the set of all vertices
on the rays of $A_\eps$ with the tree structure inherited from $T$. It is clear that
$\Teps$ is a tree and, as $A_\eps$ is compact, we have that $\partial \Teps = A_\eps$. 
\end{proof}

\begin{proof}[Proof of Theorem~\ref{smaller_tree}(a)]
We show that any one of the trees $\Teps$, $\eps>0$, satisfies the requirements of
Theorem~\ref{smaller_tree}$(a)$. Indeed, as the balls $B(v)$, $v\in \Teps_n$ are disjoint, 
we infer from Lemma~\ref{consequence_egorov}
that there can be at most $\exp{(n (f(\beta) + \delta_n))}$ vertices in $\Teps_n$. Hence,
$$\frac{1}{n} \log |\Teps_n| \leq f(\beta) + \delta_n \, . $$
Recall that $\mu^{\ssup \beta}(\partial \Teps)>0$.
Combining this with Lemma~\ref{liminf_generation_size} we obtain that 
\be{size_minimal_tree} \lim_{n \to \infty} \frac{1}{n} \log |\Teps_n| = f(\beta) \, . \ee
It remains to show that $\Teps$ supports the free energy. By (\ref{size_minimal_tree}), almost surely, 
there exists a sequence $\gamma_n \da 0$, such that for all $n \geq 1$,
\[ \frac{1}{n} \log |\Teps_n| \geq f(\beta) - \gamma_n  \, . \]
Using Lemma~\ref{consequence_egorov} again, we see that
$$\frac{1}{n} \log \Big( \sum_{v \in \Teps_n} e^{ \beta \sum_{j=1}^n V(v_j) } \Big) 
\geq \frac 1n  \log \Big(e^{n(\la'(\beta) \beta - \delta_n)} \,  |\Teps_n| \Big) 
\geq \la'(\beta) \beta - \delta_n + f(\beta)- \gamma_n,$$ 
which converges to $\la'(\beta) \beta + f(\beta) = \phi(\beta)$.
The opposite bound is trivial, hence the proof of Theorem~\ref{smaller_tree}(a) is complete.
\end{proof}

The proof of Theorem~\ref{smaller_tree}(a) immediately gives the following corollary.

\begin{corollary}\label{infGibbs}
Almost surely, for every $\beta < \beta_{\rm c}$ and for every $0<\eps <1$, there exists a tree $\Teps \subset T$ of growth rate
\[ \lim_{n\ra \infty} \frac{1}{n} \log |\Teps_n| =  f(\beta) \] such that
\[ \mu^{\ssup \beta} \{ \xi \in \partial \Teps \} \geq 1 - \eps \, . \]
\end{corollary}

We can now proceed with the second part of the proof of Theorem~\ref{smaller_tree}.


\begin{proof}[Proof of Theorem~\ref{smaller_tree}(b)] 
Since $f(\beta)$ is strictly decreasing on $(0,\beta_{\rm c})$, we can choose $\beta < \beta' < \beta_{\rm c}$ 
such that \[ \limsup_{n \ra \infty} \frac{1}{n} \log |A_n| < f(\beta') < f(\beta) \, . \]
Now, choose $\eps > 0$ small enough such that, for all $n$ sufficiently large, 
$$|A_n| \leq 
 e^{n(f(\beta') - \eps)}.$$ 
By Proposition~\ref{coarse_spectrum} we have, for large $n$, 
\be{compare_size_A_n} \# \big\{ v \in T_n \, : \, \sum_{j=1}^n V(v_j) \geq n \la'(\beta') \big\} 
\geq e^{n (f(\beta') - \eps)}  \geq |A_n| \, . \ee
Next, order the vertices $v^1, \ldots, v^{d^n}$ in the $n$th generation of $T$ such that 
$$\sum_{j=1}^n V(v_j^1) \geq \sum_{j=1}^n V(v_j^2) \geq \ldots \geq \sum_{j=1}^n V(v_j^{d^n}).$$ 
Then, clearly
$$\begin{aligned} \sum_{ v \in A_n } e^{\beta \sum_{j=1}^n V(v_j) } 
& \leq \sum_{k=1}^{|A_n|} e^{\beta \sum_{j=1}^n V(v_j^k) } \\ 
& \leq \sum_{v \in T_n} \1 \{ v \in T_n \, : \,   \sum_{j=1}^n V(v_j) \geq n \la'(\beta') \} 
\, e^{\beta \sum_{j=1}^n V(v_j)} \, ,  \end{aligned}$$
where the last inequality follows from~(\ref{compare_size_A_n}).
Note that by Lemma~\ref{max_min_H_n}, for large $n$,
$$\max_{v \in T_n} \frac1n \sum_{j=1}^n V(v_j)  \leq \la'(\beta_{\rm c}) + \eps.$$ 
Hence, we can write
$$\begin{aligned}  \sum_{v \in T_n} \1 \{ v \in T_n \, : & \,  \sum_{j=1}^n V(v_j) \geq n \la'(\beta') \} 
\, e^{\beta \sum_{j=1}^n V(v_j)} \\
& \leq \sum_{j=1}^N  \#\{ v \in T_n \, : \, \alpha_{i-1} \leq \tfrac{1}{n} \sum_{j=1}^n V(v_j) \leq \alpha_i + \eps\} \, e^{\beta n (\alpha_i + \eps) } \, , \end{aligned}$$
where $\alpha_i = (1-\frac{i}{N}) \la'(\beta') + \frac{i}{N} \la'(\beta_{\rm c})$, 
for $i =1,\ldots, N$ and some fixed $N$.
Writing $\varphi^*(\alpha) = \sup_{\tau \in \R}\{ \alpha \tau - \varphi(\tau)\}$ for the Legendre-Fenchel transform of $\varphi$, 
we find that by Proposition~\ref{coarse_spectrum} again, for $n$ sufficiently large, 
\[ \# \{ v \in T_n \, : \,  \sum_{j=1}^n V(v_j)\geq n \alpha_{i-1} \} \leq e^{n (-\phi^*(\alpha_{i-1}) + \eps)} \, . \]
Combining the previous displays and taking $N > \frac{1}{\eps}$ such that $\alpha_i \leq \alpha_{i-1} + \eps$, we obtain
\be{sum_A_n_less_than_e_sup} \begin{aligned} \sum_{ v \in A_n} e^{\beta \sum_{j=1}^n V(v_j)} & \leq \sum_{i=1}^N e^{n ( \beta \alpha_{i-1} - \phi^*(\alpha_{i-1}) + (1+2\beta) \eps)} \\
& \leq N \exp \bigg\{ n \Big(\max_{\alpha \in [\la'(\beta'), \la'(\beta_{\rm c})]} (\beta \alpha - \phi^*(\alpha)) + (1+2\beta)\eps\Big)  \bigg\} \, . \end{aligned} \ee
Since $\la'(\beta') > \me V$ it follows that
\be{max_achieved_at_beta} \max_{\alpha \in [\la'(\beta'), \la'(\beta_{\rm c})]} (\beta \alpha - \phi^*(\alpha)) \leq \max_{\alpha \in [\me V, \la'(\beta_{\rm c})]} (\beta \alpha - \phi^*(\alpha)) =
\phi(\beta) \, , \ee
where the last equality follows from the Legendre-Fenchel duality. 
But now, by~(\ref{free_energy}), $\phi = \la + \log d $ on the set $[0,\beta_{\rm c}]$ and 
therefore $\phi$ is differentiable with derivative $\la'$. Legendre-Fenchel duality implies that $\phi^*$
is strictly convex on $[\me V, \la'(\beta_{\rm c})]$. In particular, since the maximum on the right hand side
in~(\ref{max_achieved_at_beta}) is achieved at $\la'(\beta)$ and $\la'(\beta') > \la'(\beta)$, it follows that
the inequality is in fact strict.
Hence, we can choose $\eps$ small enough such that 
\[ \max_{\alpha \in [\la'(\beta'), \la'(\beta_{\rm c})]} (\beta \alpha - \phi^*(\alpha)) + 2(1+\beta) \eps <  \phi(\beta) \, .\]
Then, for $n$ large enough such that $\frac{1}{n} \log N < \eps$, we can combine the previous display with~(\ref{sum_A_n_less_than_e_sup}) to obtain the required inequality, 
\[ \frac{1}{n} \log \sum_{v \in A_n} e^{\beta \sum_{j=1}^n V(v_j)} \leq \frac{1}{n}\log N + \max_{\alpha \in [\la'(\beta'), \la'(\beta_{\rm c})]} (\beta \alpha - \phi^*(\alpha)) + (1+2\beta) \eps <  \phi(\beta) \, , \]
which completes the proof of Theorem~\ref{smaller_tree}(b).
\end{proof}

\section{Localisation in the critical regime}\label{localisation_critical_tree}

In this section, we  prove Theorem~\ref{one_ray_suffices}, in other words, we  
show that in the critical and supercritical case a single ray supports the free energy.
Recalling our convention that $\la'(\beta_{\rm c}) = \esssup V$, if $\beta_{\rm c}=\infty$, 
Theorem~\ref{one_ray_suffices} follows immediately from the following proposition.
Although the result looks similar to~(\ref{average_converges_to_alpha}), 
its proof is considerably more involved as it deals with the critical case.

\begin{prop}\label{critical_value_particle} Almost surely, there exists a ray 
$\xi \in \partial T$ such that
\[ \lim_{n \ra \infty} \frac{1}{n} \sum_{j=1}^n V(\xi_j) = \la'(\beta_{\rm c}) \, . \]
\end{prop}

The proofs in this section use ideas from branching random walks as developed in Biggins and Kyprianou~\cite{BK04} 
and Hambly et al.~\cite{HKK03}. We  split the proof in two parts according to whether $\beta_{\rm c}$ is finite
or infinite.

\subsection{Proof of Proposition~\ref{critical_value_particle} when $\beta_{\rm c} < \infty$}\label{critical_beta_finite}
 
Suppose that $f$ has a positive root, i.e. $\beta_{\rm c} < \infty$. Recall that in this case
$\la(\beta_{\rm c}) + \log d  = \beta_{\rm c} \la'(\beta_{\rm c})$, which we will use frequently
throughout this section. The idea of the proof is to restrict attention to those polymers where the average of the weights is smaller than the critical weight $\la'(\beta_{\rm c})$.
More precisely, introduce the cemetery state $\Delta$ and define new weights by
setting for $v \in T_n$ and for $x \geq 0$,
\[ \tV^x(v) = \left\{\ba{ll} V(v) & \mbox{if }  \sum_{j=1}^k V(v_j) < x + k\la'(\beta_{\rm c}) \mbox{ for all } k \leq n \, , \\
\Delta & \mbox{otherwise}\, .  \ea \right. \]
Then, it is clear that if the weight associated to $v$ is $\Delta$, then all the descendants of $v$
also have the weight $\Delta$. Moreover, for $x = 0$ we omit the superscript 
and write $\tV := \tV^0$.

The aim is now to define a martingale which induces a change of measure 
such that under the new measure there exists a ray with critical weight.
First of all, introduce a size-biased version $V^\bias$ of $V$, whose distribution is
given by $$\me \big[ g(V^\bias) \big] = \me \big[ g(V) e^{\beta_{\rm c} V - \la(\beta_{\rm c})} \big],$$ 
for any bounded, measurable function $g$. Note that 
$$\me [ V^\bias ] = \frac{\me [V e^{\beta_{\rm c} V}]}{\me [ e^{\beta_{\rm c} V}]} = \la'(\beta_{\rm c}).$$ 
Therefore, if $(V^\bias_j, j \geq 1)$ is a sequence of independent random variables
with the same distribution as $V^\bias$, then the random walk with increments given by 
$(V^\bias_j)$ has a drift $\la'(\beta_{\rm c})$. 
Now, define $\tau = \inf\{ n \geq 1 \,:\, \sum_{j=1}^n V^\bias_j < n  \la'(\beta_{\rm c})\}$
as the first time that the random walk with increments $(V^\bias_j)$ grows
slower than its drift. 
Then, we set for $x>0$,
\[ \specialh(x) = \me \bigg[ \sum_{i=0}^\tau \1 \Big\{ \sum_{j=1}^n V^\bias_j -n  \la'(\beta_{\rm c})\in [0,x) \Big\}
\bigg] \, , \]
as the expected number of visits of the normalised random walk with increments $(V^\bias_j - \la'(\beta_{\rm c}))$
to $[0,x)$ before hitting $(-\infty,0)$. Furthermore, we set $\specialh(0) = 1$. 

For $x \geq 0$, we define the martingale $(W_n^x \colon n\ge 0)$ by
\[ W_n^x = \sum_{v\in T_n} \frac{\specialh(x - \sum_{j=1}^n V(v_j) + n\la'(\beta_{\rm c}))}{\specialh(x)} 
\, e^{\beta_{\rm c} \sum_{j=1}^n V(v_j) - n(\la(\beta_{\rm c})+\log d )} \, \1_{\{\tV^x(v) \neq \Delta\}} \, . \]
Again, for $x=0$ we omit the superscript and write $W_n = W_n^0$. 
In order to prove that this defines a martingale, we need the following facts, see Lemma~10.1 in~\cite{BK04}. 

\begin{lemma}\label{properties_of_V} \ \\[-5mm]
\begin{enumerate}
	\item[(i)] As $x \ra \infty$, $\frac{\specialh(x)}{x} \ra C$, for some constant $C >0$.
	\item[(ii)] For $x \geq 0$, we have $\me \big[ \specialh(x - V^* + \la'(\beta_{\rm c})) \1\{ x - V^* + \la'(\beta_{\rm c}) > 0 \} \big] = \specialh(x) $.
\end{enumerate}
\end{lemma}

Now, the proof that $(W_n^x \colon n\ge 0)$ is a martingale with respect to the filtration given
by $\calF_n = \sigma(V(v)\, :\, |v| \leq n )$ is a straight-forward calculation.

\begin{lemma}\label{W_n_is_martingale} The process $(W_n^x \colon n\geq 1)$ defines a martingale of mean one.
\end{lemma}

\begin{proof} 
Recall that $\la(\beta_{\rm c}) + \log d  = \beta_{\rm c} \la'(\beta_{\rm c})$. 
Then
\[\begin{aligned}  \specialh(x) & \E [ W^x_{n+1}  \, | \, \calF_n ]  = \E \Big[ \sum_{v \in T_{n+1}} \!\!\!\specialh\Big(x - \sum_{j=1}^{n+1}( V(v_j) - \la'(\beta_{\rm c}))\Big) e^{\beta_{\rm c} \sum_{j=1}^{n+1} (V(v_j) -\la'(\beta_{\rm c}))} \1{\big\{\tV^x(v) \neq \Delta\big\}} \, \Big| \, \calF_n \Big] \\
& = \sum_{\substack{v \in T_{n}\\ \tV^x(v) \neq \Delta}}
 \sum_{w \in T_1(v)} \E \Big[ \specialh\Big(x - \sum_{j=1}^n V(v_j) - V(w) + (n+1)\la'(\beta_{\rm c}) \Big)\\ 
 & \hspace{1.7cm} \times e^{\beta_{\rm c} (\sum_{j=1}^{n} V(v_j) + V(w) - (n+1)\la'(\beta_{\rm c}))} 
 \1\big\{{\textstyle \sum_{j=1}^n V(v_j) + V(w)  < x + (n+1)\la'(\beta_{\rm c})} \big\}  \, \Big| \, \calF_n \Big].
\end{aligned} \]
Now note that $V(w)$ is independent of $\calF_n$ and recall the definition of $V^*$. 
Then we can continue the display with
\[\begin{aligned} 
& = \sum_{\substack{ v \in T_{n} \\ \tV^x(v) \neq \Delta} } e^{\beta_{\rm c} ( \sum_{j=1}^n V(v_j) - n \la'(\beta_{\rm c}))} \me\Big[ \specialh\Big( x - \sum_{j=1}^n V(v_j) + n \la'(\beta) - V^* + \la'(\beta_{\rm c}) \Big) \\[-0.5cm] &   \hspace{5cm}\times\1{\big\{{\textstyle x - \sum_{j=1}^n V(v_j) + n \la'(\beta_{\rm c})   - V^* +\la'(\beta_{\rm c}) > 0 }\big\}}\Big] \\
& = \sum_{\substack{ v \in T_{n} \\ \tV^x(v) \neq \Delta}} \specialh\Big(x - \sum_{j=1}^n V(v_j) + n \la'(\beta_{\rm c})\Big) e^{\beta_{\rm c}( \sum_{j=1}^n V(v_j) - n \la'(\beta_{\rm c})) } = \specialh(x) W_{n}^x \, , \end{aligned} \]
where we have used Lemma~\ref{properties_of_V}~(ii). This lemma also confirms that $W_n^x$ has mean $1$.
\end{proof}


 
Allowing the cemetery state as a possible weight in $\st$ we can, similarly as in Section~\ref{subsection_ergodic_theory}, extend the measure $\p$ to a measure $\p^*$ on $\st$ 
by choosing the spine uniformly, i.e. by choosing $\xi_{n+1}$ with equal probability from the 
children of $\xi_n$. Define the extended filtration
\[ \calF^*_n = \sigma( \calF_n, \, \xi_i , i = 1,\ldots,n) \, . \]
We now perform a change of measure such that the weights $(V(\xi_i))$ along the spine will be chosen such that
$\sum_{j=1}^n (V(\xi_j) - \la'(\beta_{\rm c}))$ follows the law of a random walk conditioned to stay positive. 
More precisely, define the probability measure $\Q^*$ via
\[ \left.\frac{d \, \Q^*}{d \, \p^*}\right|_{\calF^*_n} = \specialh\Big( n \la'(\beta_{\rm c}) - \sum_{j=1}^n V(\xi_j) \Big) \,  e^{\beta_{\rm c} \sum_{j=1}^n V(\xi_j) - n \la(\beta)} \1\{ \tV(\xi_n) \neq \Delta \} \, . \]
From the definition it follows that under the new measure $\Q^*$, the distribution of the weights is constructed as follows:
\begin{itemize}
\item The spine $\xi$ is chosen uniformly, i.e. $\xi_{n+1}$ is chosen uniformly among the 
children~of~$\xi_n$. 
	\item The weights along the spine $\xi$ are distributed such that their average is conditioned to be 
	less than the critical weight $\la'(\beta_{\rm c})$, i.e.  if at time $n$ the weights along the spine satisfy $s = \sum_{j=1}^n V(\xi_j) < n \la'(\beta_{\rm c})  $, then the weight for $\xi_{n+1}$ is chosen according to Doob's $h$-transform,
	\[ \begin{aligned} \Q^* & \Big[ V(\xi_{n+1}) \in dz \, \Big| \, \sum_{j=1}^n V(\xi_j) =s \Big] \\ 
	&= \frac{\specialh((n+1)\la'(\beta_{\rm c}) -(z+s))}{\specialh(n\la'(\beta_{\rm c})-s)} \, \1\{z +s < (n+1)\la'(\beta_{\rm c}) \}\, e^{\beta_{\rm c} z - \la(\beta_{\rm c})} \,\mp \{ V \in dz\} \, . \end{aligned}\]
	\item The weights of the vertices not on the spine remain unaffected by the change of measure. In other words, if $\eta_n$ is a sibling of $\xi_n$, then we generate a weight $V(\eta_n)$ with the distribution of~$V$ and attach it 
	to $\eta_n$ if $$\sum_{j=1}^{n-1}V(\xi_j) + V(\eta_n) < n \la'(\beta'),$$ and otherwise $\eta_n$ receives the weight~$\Delta$. Then conditionally on $\tV(\eta_n) \neq \Delta$, the random disorder in the tree started in $\eta_n$ 
is given by the weights $(\tV^x(v)\, : \, v \in T(\eta_n))$ for 
$$x = n \la'(\beta_{\rm c}) - \sum_{j=1}^{n-1}V(\xi_j) - V(\eta_n).$$
\end{itemize}

If we restrict $\Q^*$ to the $\sigma$-algebra $\calF = \sigma(\bigcup_{n\geq 1}\calF_n)$, we obtain a measure $\Q$ defined on the space of trees with weights. Moreover, we obtain its density on ${\mathcal F}_n$.

\begin{lemma}\label{restricted_Q}
\[ \left.\frac{d \, \Q }{d \, \p }\right|_{\calF_n} = W_n \, . \]
\end{lemma}

\begin{proof} Writing $\p^*[\,\cdot\,]$ for the expectation with respect to $\p^*$, 
we obtain from the definition of conditional expectation
\[ \begin{aligned} \left.\frac{d \, \Q^*}{d\, \p^*}\right|_{\calF_n} & = \p^* \Big[ 
\specialh\Big(n \la'(\beta_{\rm c})- \sum_{j=1}^n V(\xi_j)\Big) \,  e^{\beta_{\rm c} \sum_{j=1}^n V(\xi_j) - n \la(\beta)} \1\{ \tV(\xi_n) \neq \Delta \}  \, \Big| \, \calF_n \Big] \\
& = \p^* \Big[ \sum_{v \in T_n} \1\{\xi_n = v\} \,\,\specialh\Big(n \la'(\beta_{\rm c})- \sum_{j=1}^n V(v_j)\Big) \,  e^{\beta_{\rm c} \sum_{j=1}^n V(v_j) - n \la(\beta)} \1\{ \tV(v) \neq \Delta \}  \, \Big| \, \calF_n \Big] \\
 & = \sum_{\substack{ v \in T_n \\ \tV(v)	 \neq \Delta}}   \specialh\Big(n \la'(\beta_{\rm c})- \sum_{j=1}^n V(v_j)\Big) \,  e^{\beta_{\rm c} \sum_{j=1}^n V(v_j) - n \la(\beta)} \, \p^*\{\xi_n = v\} = W_n \, , \end{aligned}\]
which proves the claim.\end{proof}

The next step will be to show that $\Q$ is absolutely continuous with respect to $\p$. 

\begin{lemma}\label{absolute_continuous} $\Q$ is absolutely continuous with respect to $\p$ with 
Radon-Nikod\'ym derivative 
$$W := \limsup_{n \ra \infty} W_n.$$ Furthermore, $\Q^*$-almost surely
\[ \lim_{n \ra \infty} \frac{1}{n} \sum_{j=1}^n V(\xi_j) = \la'(\beta_{\rm c}) \, . \]
\end{lemma}

\begin{proof} By a standard measure theoretic result, see for instance Lemma 11.2. in~\cite{LP05}, 
\be{absolute_continuity} \frac{d \, \Q}{d \,\p} = W \iff W < \infty \quad \Q\mbox{-almost surely} \, . \ee
Denote by $\calG = \sigma(V(\xi_k) \, :\, k = 1,2,\ldots)$ the $\sigma$-algebra containing all the 
information about the weights along the spine. 
The first step is to calculate the conditional expectation $\Q^*[ W_n \, |\, \calG]$.
With this in mind, consider a path $v \in T_n$. Decomposing according to the last common 
ancestor with the spine,
\[ \begin{aligned} \Q^* \Big[ & \, \specialh\Big(n  \la'(\beta_{\rm c}) -  \sum_{j=1}^n V(v_j)\Big) \, e^{\beta_{\rm c} \sum_{j=1}^n V(v_j) - n(\la(\beta_{\rm c})+\log d )} \, \1\{ \tV(v) \neq \Delta\} \, \Big| \, \calG \Big]\\
& = \sum_{m=0}^{n} \specialh\Big(m  \la'(\beta_{\rm c}) - \sum_{j=1}^m V(\xi_j)\Big) \, e^{\beta_{\rm c} \sum_{j=1}^m V(\xi_j) - m (\la(\beta_{\rm c})+\log d )  } \, \Q^*\{ \max\{ k \colon v_k = \xi_k \} = m \}
\\ & \qquad \times \E\bigg[ \prod_{i=m+1}^n  \tfrac{\specialh(i  \la'(\beta_{\rm c}) -  \sum_{j=1}^i V(v_j)}{\specialh( (i-1)  \la'(\beta_{\rm c}) -  \sum_{j=1}^{i-1} V(v_j)} \, e^{\beta_{\rm c} V(v_i) - \la(\beta_{\rm c})} \, \1_{\{ \sum_{j=1}^i V(v_j) < i \la'(\beta_{\rm c})\}} \, \bigg| \, \calF_{m} \bigg] \\
& \le \sum_{m=0}^{n}\specialh\Big(m  \la'(\beta_{\rm c}) - \sum_{j=1}^m V(\xi_j)\Big) \, d^{\,-n}
\,  e^{\beta_{\rm c} \sum_{j=1}^m V(\xi_j) - m (\la(\beta_{\rm c})+\log d )  }  \, , \end{aligned} \]
where we used the fact that under $\Q^*$ the weights of the vertices not on the spine have the same distribution as under $\p^*$, so that we can apply Lemma~\ref{properties_of_V}~(ii) repeatedly to show that the conditional expectation of the product is equal to $1$. Summing over all $v\in T_n$ we obtain from the previous equation 
\[ \begin{aligned} \Q^* [  W_n \, | \, \calG ] &  \le \sum_{m=0}^{n}\specialh\Big(m  \la'(\beta_{\rm c}) - \sum_{j=1}^m V(\xi_j)\Big) \,   e^{\beta_{\rm c} \sum_{j=1}^m V(\xi_j) - m (\la(\beta_{\rm c})+\log d )  }
 \, . \end{aligned} \]
Recall that $\sum_{j=1}^n V(\xi_j)$ under $\Q^*$ has the law of a random walk conditioned to stay strictly below $n \la'(\beta_{\rm c})$. 
In other words, $- \sum_{j=1}^n V(\xi_j) + n \la'(\beta_{\rm c})$ follows the law of a random walk conditioned to stay positive. It is known, see for instance~\cite{HKK03} where they treat the case of a random walk conditioned to stay non-negative, that $\Q^*$-almost surely  for any $\eps >0$, there exist constants $C_1,C_2 > 0$ such that for all sufficiently large $n$,
\be{growth_of_RW} C_1 n^{\frac{1}{2} - \eps} \leq - \sum_{j=1}^n V(\xi_j) + n \la'(\beta_{\rm c}) \leq C_2 n^{\frac{1}{2}+\eps} \, . \ee
Hence, using that by Lemma~\ref{properties_of_V}, $h(x)/x \ra C$ as $x \ra \infty$,  the previous estimate shows that, $\Q^*$-almost surely 
\[ \limsup_{n \ra \infty} \Q^*[ W_n \, | \, \calG] < \infty \, . \]
By Fatou's lemma we can conclude that $\liminf_{n \ra \infty}W_n$ is also $\Q^*$-almost surely  finite, so in particular it is $\Q$-almost surely  finite. From the representation in Lemma~\ref{restricted_Q}, we see that $1/W_n$ is a nonnegative super-martingale under $\Q$ and hence it has a $\Q$-almost sure limit. Hence, $\Q$-almost surely  $W =\limsup_{n \ra \infty} W_n = \liminf_{n \ra \infty} W_n < \infty$, so that by~(\ref{absolute_continuity}), $\Q$ is absolutely continuous with respect to $\p$ with Radon-Nikod\'ym derivative $W$.
Moreover, (\ref{growth_of_RW}) shows that $\Q^*$-almost surely
\[ \lim_{n \ra \infty} \frac{1}{n} \sum_{j=1}^n V(\xi_j) = \la'(\beta_{\rm c}) \, . \]\\[-1cm]
 \end{proof}

Now, we are finally in the position to complete the proof of Proposition~\ref{critical_value_particle}.

\begin{proof}[Proof of Proposition~\ref{critical_value_particle} when $\beta_{\rm c} < \infty$] By Lemma~\ref{absolute_continuous}, we know that  $\Q^*$-almost surely, the weights along the spine satisfy
\be{spine_critical_value} \lim_{n \ra \infty} \frac{1}{n} \sum_{j=1}^n V(\xi_j) = \la'(\beta_{\rm c}) \, . \ee
Now projecting down onto $\calF$, we see that $\Q$-almost surely  there exists a ray $\xi \in \partial T$ that satisfies~(\ref{spine_critical_value}). But since $\Q$ is absolutely continuous with respect to $\p$, we can deduce that 
\[ \p \Big\{ \mbox{ there exists } \xi \in \partial T \mbox{ with } \lim_{n\ra\infty}\tfrac{1}{n} \sum_{j=1}^n V(\xi_j) = \la'(\beta_{\rm c}) \Big\} > 0 \, . \]
But the event in question is a tail event with respect to the i.i.d.~family of weights,  so that by Kolmogorov's zero-one law it follows that the event has probability $1$. 
\end{proof}

\subsection{Proof of Proposition~\ref{critical_value_particle} when $\beta_{\rm c} = \infty$}\label{critical_beta_infinite}

We now consider the case that $f$ does not have a positive root.
By Lemma~\ref{criterion_for_root_of_f} this implies that $w = \esssup V$ is finite and
$\mp \{ V = w \} \geq \frac{1}{d}$. 
We start by considering the special case of a Bernoulli disorder. 
Therefore, assume that $\mp \{ V =1 \} = p = 1-\mp \{ V = 0\}$
with $p \geq \frac1d$.
At the end of this
section we will see that it is easy to generalize the result and 
to prove Proposition~\ref{critical_value_particle}
for a general disorder with~$\beta_{\rm c} = \infty$.

\begin{lemma}\label{critical_ray_V_binary} For the Bernoulli disorder with success probability $p \geq \frac{1}{d}$, almost surely, there exists a ray $\xi \in \partial T$ such that
\[ \lim_{n \ra \infty} \frac{1}{n} \sum_{j=1}^n V(\xi_j) = 1 \, . \]
\end{lemma}

As in the previous Section~\ref{critical_beta_finite} we  use a change of measure
argument. In this case, our aim is to produce a new measure under which the spine
has an asymptotic average weight equal to $1$.

\begin{proof}
Fix $p \in [\frac{1}{d},1)$. Define a sequence $(p_i)_{i \geq 1}$ of increasing numbers in $[p,1)$ that converges
to $1$ by setting 
$p_i = \max\{(\frac{1}{i})^{2/i}, p \}$.
As before, let $\p$ be the probability measure such that the random variables $(V(v) \, : \, v \in T)$
are independent random variables with Bernoulli distribution with success probability $p$.
Next, we extend $\p$ to a probability measure $\p^*$ on the set of spined trees such
that the spine is chosen uniformly. Also, set $\calF^*_n = \sigma(V(v) , |v| \leq n, \xi(j), j \leq n)$
and denote its projection onto the trees with random weights by $\calF_n = \sigma(V(v) , |v| \leq n)$.
Then, we can define a new probability 
measure $\Q^*$ on the set of spined trees by setting 
\[ \left.\frac{d\, \Q^*}{d\, \p^*}\right|_{\calF^*_n} = \prod_{i=1}^n \bigg(\frac{p_i}{p}\bigg)^{V(\xi_i)} 
\bigg(\frac{1-p_i}{1-p}\bigg)^{1-V(\xi_i)}
\, . \]
It is easy to check that the right hand side defines a martingale under $\p^*$, which implies
that the measure $\Q^*$ is well-defined. Moreover, under the new measure the 
spine $\xi$ is still chosen uniformly, but $V(\xi_i)$ is now Bernoulli
with success probability $p_i$, whereas  if $v \neq \xi_i$,
for any $i$, $V(v)$ is still Bernoulli with success probability $p$.

Now, we can define $\Q$ as the projection of $\Q^*$ onto 
$\calF = \sigma( \bigcup_{n \geq 1} \calF_n)$. Then, as before
\[\begin{aligned} \left.\frac{d \,\Q^*}{d\, \p^*}\right|_{\calF_n} & = 
\p^* \Big[\prod_{i=1}^n \left(\frac{p_i}{p}\right)^{V(\xi_i)} 
\left(\frac{1-p_i}{1-p}\right)^{1-V(\xi_i)} \, \Big| \, \calF_n \Big]\\
& = \p^* \Big[ \sum_{v\in T_n} \1_{\{v = \xi_n\}} \prod_{i=1}^n \left(\frac{p_i}{p}\right)^{V(v)} 
\left(\frac{1-p_i}{1-p}\right)^{1-V(v)} \, \Big| \, \calF_n \Big] \\
& = \sum_{v\in T_n} \frac{1}{d^n} \prod_{i=1}^n \left(\frac{p_i}{p}\right)^{V(v)} 
\left(\frac{1-p_i}{1-p}\right)^{1-V(v)} =: M_n \, . \end{aligned} \]
Clearly, $(M_n, n \geq 0)$ defines a martingale with respect to $\p$ and 
the filtration $(\calF_n, n \geq 0)$. As in the proof of Lemma~\ref{absolute_continuous},
our aim will be to show that $M = \limsup_{n \ra \infty} M_n < \infty$, $\Q$-almost surely. 
For this purpose define $\calG = \sigma( V(\xi_i) \, : \, i \geq 1)$ and consider the conditional expectation
\[ \begin{aligned} \Q^* [ M_n \, | \, \calG ] & =
\Q^* \Big[ \sum_{v\in T_n} \sum_{m=0}^n \1_{\{ \max\{ k : v_k = \xi_k \} = m\}}
\frac{1}{d^n}  \prod_{i=1}^n \left(\frac{p_i}{p}\right)^{V(v)} \left(\frac{1-p_i}{1-p}\right)^{1-V(v)} 
\, \Big| \, \calG \Big]\\
& = \sum_{m=0}^n \prod_{i=1}^m \left(\frac{p_i}{p}\right)^{V(\xi_i)} \left(\frac{1-p_i}{1-p}\right)^{1-V(\xi_i)}
\frac{1}{d^n} \, \# \{ v \in T_n \, : \, \max\{ k : v_k = \xi_k \} = m \} \\
& \le \sum_{m=0}^n \frac{1}{d^m} \prod_{i=1}^m \left(\frac{p_i}{p}\right)^{V(\xi_i)} \left(\frac{1-p_i}{1-p}\right)^{1-V(\xi_i)}
\, . 
\end{aligned} \] 
Now, recall that $p_i \geq p$ so that $\frac{1-p_i}{1-p\ }\leq \frac{p_i}{p}$. 
Hence, using that $p_i$ is increasing and $p \geq \frac{1}{d}$, we can deduce 
from the previous display that
\[ \Q^* [ M_n \, | \, \calG ] 
\leq \sum_{m=0}^n \frac{1}{d^m} p^{-m} \prod_{i=1}^m p_i
\leq \sum_{m=0}^n p_m^m  \, . \]
Hence, $\limsup_{n \ra \infty} \Q^* [ M_n \, | \, \calG ] < \infty$, since $p_m^m = \frac{1}{m^2}$ for all $m$
large enough.
Precisely, as in Section~\ref{critical_beta_finite} we can thus deduce by Fatou's lemma
that $\liminf_{n \ra \infty} M_n $ is $\Q^*$-almost surely  finite and thus $\Q$-almost surely  finite. 
By construction, $\frac{1}{M_n}$ is a positive $\Q$-martingale, which implies that its limit exists and
hence $M = \lim_{n \ra \infty} M_n < \infty$, $\Q$-almost surely.  Therefore, 
$\Q$ is absolutely continuous with respect to $\p$ with Radon-Nikod\'ym derivative $M$.

We have seen that $\Q^*[V(\xi_i)] = p_i$. 
Since $p_i \ra 1$ as $i \ra \infty$,  it is clear that $\lim_{n \ra \infty} \frac{1}{n} \Q^*[ \sum_{j=1}^n V(\xi_j) ]
= \lim_{n \ra \infty} \frac{1}{n} \sum_{j=1}^n p_j = 1$. Now $0 \leq V(\xi_i) \leq 1$ so that by Lebesgue's
dominated convergence theorem, 
\[ \Q^* \Big[ \limsup_{n \ra \infty} \Big( 1- \frac{1}{n} \sum_{j=1}^n V(\xi_j) \Big) \Big]
\le 1 -\lim_{n \ra \infty} \Q^* \Big[\frac{1}{n} \sum_{j=1}^n V(\xi_j)\Big] = 0 \, . \]
Since $0\le 1- \frac{1}{n} \sum_{j=1}^n V(\xi_j) \le 1$, we deduce that 
$\Q^*$-almost surely  $\lim_{n\ra \infty} \frac{1}{n} \sum_{j=1}^n V(\xi_j) = 1$.

Hence, $\Q$-almost surely,  there exists a ray $\xi \in\partial T$ such that
$\lim_{n\ra \infty} \frac{1}{n} \sum_{j=1}^n V(\xi_j) = 1$. As $\Q$
is absolutely continuous with respect to $\p$ it follows that 
\[ \p \Big\{ \mbox{ there exists } \xi \in \partial T \mbox{ with } 
\lim_{n \ra \infty} \frac{1}{n} \sum_{j=1}^n V(\xi_j) = 1 \Big\} > 0 \, . \]
As in the previous section, we deduce from Kolmogorov's zero-one law that
this probability is in fact equal to $1$, so that we have proved Lemma~\ref{critical_ray_V_binary}.
\end{proof}

We now use the previous lemma for the Bernoulli disorder to complete the proof of Proposition~\ref{critical_value_particle}. Assume that $V$ is any random variable 
such that the corresponding function $f$ has no positive root. Recall that this means
that $\mp \{ V = w \} \geq \frac{1}{d}$ for $w = \esssup V<\infty$.

\begin{proof}[Proof of Proposition~\ref{critical_value_particle} when $\beta_{\rm c} = \infty$]
Given the disorder $(V(v), v \in T)$, define the random variables $\tilde{V}(v) = \1\{ V(v) = w \}$. Then
$p := \p\{ \tilde{V}(v) = 1 \} = \p\{ V(v) = w \} \geq \frac{1}{d}$.  
Lemma~\ref{critical_ray_V_binary} shows that there exists a ray $\xi \in \partial T$ such that
\[ \lim_{n \ra \infty} \frac{1}{n} \sum_{j=1}^n \tilde{V}(\xi_j) = 1 \, . \]
Therefore, 
\[ \liminf_{n \ra \infty} \frac{1}{n} \sum_{j=1}^n V(\xi_j) \geq  \liminf_{n \ra \infty} \frac{1}{n} \sum_{j=1}^n V(\xi_j)
\1\{ V(\xi_j) = w \} \geq w \liminf_{n \ra\infty}\frac{1}{n} \sum_{j=1}^n
\tilde{V}(\xi_j)  = w \, . \]
Finally, as the reversed inequality is trivial, we have completed the proof. 
\end{proof}

\section{$\varrho$\,-percolation on regular trees}\label{rho_percolation}

We now show how the directed polymer model on trees can be interpreted in the framework of 
$\varrho$-percolation. Consider a $d$-ary tree $T$ as before and, for $p \in [0,1]$, define the 
disorder ${\mathcal V}_p=(V_p(v) \colon v \in T)$ as a family of i.i.d. Bernoulli random variables with 
success parameter $p$. An edge leading to a vertex $v$ with weight $V_p(v) = 1$ is considered 
to be open and if $V_p(v) = 0$ it is defined to be closed. For $\varrho \in [p,1]$, we say
that $\varrho$-percolation occurs if there exists a path $\xi \in \partial T$ such that
\[ \liminf_{n \ra \infty} \frac{1}{n} \sum_{j=1}^n V_p(\xi_j) \geq \varrho \, . \]

\begin{lemma}\label{fixed_p}
Fix $p\in(0,1)$ and  let $\la_p(\beta) = \log \me [ e^{\beta V_p} ]$. Let
$\alpha_{\rm c}(p)=1$, if  $p \geq \frac{1}{d}$, and otherwise let $\alpha_{\rm c}(p)$
be the unique solution of $\la_p^*(\alpha)=\log d$ in the interval $(p,1)$.
Then, if $\alpha \leq \alpha_{\rm c}(p)$, almost surely, there exists $\xi \in \partial T$ such that 
$$\liminf_{n \ra \infty} \frac{1}{n} \sum_{j=1}^n V_p(\xi_j) 
\geq \alpha,$$
but if $\alpha > \alpha_{\rm c}(p)$ almost surely no such $\xi \in \partial T$ exists.
\end{lemma}

\begin{proof}
Using Lemma~\ref{criterion_for_root_of_f} we see that the critical parameter $\beta_{\rm c}=\beta_{\rm c}(p)$ for the
polymer model with disorder ${\mathcal V}_p$ is infinite if $p \geq \frac{1}{d}$ and finite otherwise.
In the latter case, this implies that $\alpha_{\rm c}(p)$ is well-defined and $\alpha_{\rm c}(p) = \la_p'(\beta_{\rm c})$. 
From Proposition~\ref{critical_value_particle} we hence obtain in both cases
that there exists a ray $\xi \in \partial T$ 
such that \[ \lim_{n\ra \infty} \frac{1}{n} \sum_{j=1}^n V_p(\xi_j) = \alpha_{\rm c}(p) \, . \]
To show that there is no ray~$\xi\in\partial T$ along which we obtain a larger liminf, 
we may assume that $p < \frac{1}{d}$. Recall that, by~(\ref{free_energy}), 
the free energy
\[ \phi_p(\beta) = \lim_{n \ra \infty} \frac1n \log \sum_{v\in T_n} e^{\beta \sum_{j=1}^n V_p(v_j) }\]
satisfies $\phi_p(\beta_{\rm c}) = \beta_{\rm c} \alpha_{\rm c}(p)$. Hence, for any ray $\xi \in \partial T$,
\[ \liminf_{n \ra \infty} \frac{1}{n} \sum_{j=1}^n V_p(\xi_j) \leq 
\frac{1}{\beta_{\rm c}} \liminf_{n \ra \infty}  \frac{1}{n} \log \sum_{v\in T_n} e^{\beta_{\rm c} \sum_{j=1}^n V_p(v_j) } 
= \frac{\phi_p(\beta_{\rm c})}{\beta_{\rm c}}
= \alpha_{\rm c}(p) \, , \]
which proves the second part of Lemma~\ref{fixed_p}.
\end{proof}

As the next step, we give an explicit formula for $\alpha_{\rm c}(p)$ when $p < \frac{1}{d}$. 
First, we compute the logarithmic moment generating function and its derivative
\[ \la_p(\beta) = \log \me [ e^{\beta V_p}] = \log (p e^\beta + (1-p)) \quad \mbox{and} \quad \la_p'(\beta) = \frac{p e^\beta}{p e^\beta + (1-p)} \, . \]
Then, using that $\alpha_{\rm c}(p)=\la'_p(\beta_{\rm c}(p))$ 
for the polymer with disorder~${\mathcal V}_p$, we get
\begin{equation}\label{alpha}
\alpha_{\rm c}(p)= \frac{p e^{\beta_{\rm c}(p)}}{p e^{\beta_{\rm c}(p)} + (1-p)}.
\end{equation}
As $\log d=\lambda^*(\alpha_{\rm c}(p))=\alpha_{\rm c}(p)\beta_{\rm c}(p)-\log (p e^{\beta_{\rm c}(p)} + (1-p))$, we obtain
\be{alpha_p_critical} p^{\alpha_{\rm c}(p)} (1-p)^{1-\alpha_{\rm c}(p)} d = \alpha_{\rm c}(p)^{\alpha_{\rm c}(p)} (1-{\alpha_{\rm c}(p)})^{1-{\alpha_{\rm c}(p)}} \, . \ee
It is easy to see from Lemma~\ref{fixed_p} that $\alpha_{\rm c}(\,\cdot\,)$ is an increasing function 
on $(0,\frac1d]$, and from~\eqref{alpha_p_critical} that it is strictly increasing. 

To complete the proof of Theorem~\ref{thm_rho_percolation} fix $\varrho \in (0,1]$. First note (by taking 
the derivative) that the function $g(p)=p^\varrho (1-p)^{1-\varrho}$ is strictly increasing on the interval 
$(0,\varrho]$  so that there is indeed a unique solution to the equation characterising~$p_{\rm c}$. In the 
special case $\varrho=1$ this solution is given by $p_{\rm c}=\frac 1d$. Back to the general case, by \eqref{alpha_p_critical} we have $\varrho = \alpha_{\rm c}(p_{\rm c})$. 
This value $p_{\rm c}$ is indeed the critical parameter, since if $p \geq p_{\rm c}$ we have $\alpha_{\rm c}(p) \geq \alpha_{\rm c}(p_{\rm c}) = \varrho$ so that $\varrho$-percolation occurs by Lemma~\ref{fixed_p}. Moreover, if $p < p_{\rm c}$, then $\alpha_{\rm c}(p) < \alpha_{\rm c}(p_{\rm c}) = \varrho$ 
so that $\varrho$-percolation does not occur, which completes the proof of Theorem~\ref{thm_rho_percolation}.


\end{document}